\documentclass[11pt]{amsart}

\usepackage{epsfig}
\usepackage{amsmath, amsthm, amssymb}

\newtheorem{defn}{Definition}[section]

\newtheorem{cor}[defn]{Corollary}
\newtheorem{thm}[defn]{Theorem}

\newtheorem{lemma}[defn]{Lemma}

\newtheorem{remark}[defn]{Remark}

\newtheorem{cond}[defn]{Condition}

\newcommand{\be}{\begin{equation}}
\newcommand{\ee}{\end{equation}}
\newcommand{\bea}{\begin{eqnarray}}
\newcommand{\eea}{\end{eqnarray}}
\newcommand{\beas}{\begin{eqnarray*}}
\newcommand{\eeas}{\end{eqnarray*}}

\newcommand{\R}{\mathbb{R}}

\newcommand{\ve}{\varepsilon}

\newcommand{\noi}{\noindent}

\newcommand{\goto}{\rightarrow}
\newcommand{\hsp}{\hspace{.3in}}
\newcommand{\bp}{\begin{proof}}
\newcommand{\ep}{\end{proof}}

\newcommand{\TV}{\tiny{\mbox{TV}}}
\newcommand{\mix}{\tiny{\mbox{mix}}}

\begin{document}

\title[Rapid Mixing of Glauber Dynamics of Gibbs Ensembles via  Aggregate Path Coupling]{Rapid Mixing of Glauber Dynamics of Gibbs Ensembles via  Aggregate Path Coupling and Large Deviations Methods}

\author{Yevgeniy Kovchegov}
\address{Department of Mathematics, Oregon State University, Corvallis, OR  97331}
\email{kovchegy@math.oregonstate.edu}

\author{Peter T. Otto}
\address{Department of Mathematics, Willamette University, Salem, OR 97302,  phone: 1-503-370-6487}
\email{potto@willamette.edu}


\begin{abstract}
In this paper, we present a novel extension to the classical path coupling method to statistical mechanical models which we refer to as aggregate path coupling.  In conjunction with large deviations estimates, we use this aggregate path coupling method to prove rapid mixing of Glauber dynamics for a large class of statistical mechanical models, including models that exhibit discontinuous phase transitions which have traditionally been more difficult to analyze rigorously.  The parameter region for rapid mixing for the generalized Curie-Weiss-Potts model is derived as a new application of the aggregate path coupling method.
\end{abstract}

\date{\today}

\maketitle

\section{Introduction} 

\medskip

In recent years, mixing times of dynamics of statistical mechanical models have been the focus of much probability research, drawing interest from researchers in mathematics, physics and computer science.  The topic is both physically relevant and mathematically rich.  But up to now, most of the attention has focused on particular models including rigorous results for several mean-field models.  A few examples are (a) the Curie-Weiss (mean-field Ising) model \cite{DLP1, DLP2, EL, LLP}, (b) the mean-field Blume-Capel model \cite{EKLV, KOT}, (c) the Curie-Weiss-Potts (mean-field Potts) model \cite{BR, CDLLPS, GSV}.  A good survey of the topic of mixing times of statistical mechanical models can be found in the recent paper by Cuff et.\ al.\ \cite{CDLLPS}.

\medskip

An important question driving the work in the field is the relationship between the mixing times of the dynamics and the equilibrium phase transition structure of the corresponding statistical mechanical models.  For example, the Curie-Weiss model, which undergoes a continuous, second-order, phase transition, was one of the first models studied to investigate this relationship and it was found that the mixing times undergo a transition from rapid to slow mixing at precisely the same critical value as the equilibrium phase transition \cite{LPW}.  This property was also shown for the mean-field Blume-Capel model \cite{KOT} in the parameter regime where the model undergoes a continuous, second-order phase transition.  

\medskip

On the other hand, for models that exhibit a discontinuous, first-order, phase transition, they do not appear to share this same property.  This was first verified for the mean-field Blume-Capel model in the discontinuous phase transition parameter regime \cite{KOT} and recently for the Curie-Weiss-Potts model \cite{CDLLPS}.  For these models, it was shown that the change in mixing times occurs, not at the equilibrium phase transition value, but instead at a smaller parameter value at which metastable states first emerge.  

\medskip

The results for models that exhibit a continuous phase transition were obtained by a direct application of the standard path coupling method that requires contraction of the mean path coupling distance between all neighboring configurations. See \cite{BD} and \cite{LPW}.  For models that exhibit a discontinuous phase transition straightforward path coupling methods fail and the results of \cite{KOT} were obtained by applying a novel extension called {\it aggregate path coupling} in one dimension and large deviations estimates.

\medskip

In this paper, we extend the work we did in \cite{KOT} and provide a single general framework for determining the parameter regime for rapid mixing of the Glauber dynamics for a large class of statistical mechanical models, including all those listed above.  The aggregate path coupling method presented here extends the classical path coupling method in two directions.  First, we consider macroscopic quantities in higher dimensions and find a monotone contraction path by considering a related variational problem in the continuous space. We also do not require the monotone path to be a nearest-neighbor path. In fact, in most situations we consider, a nearest-neighbor path will not work for proving contraction.  Second, the aggregation of the mean path distance along a monotone path  is shown to contract for some but not all pairs of configurations.  Yet, we use measure concentration and large deviation principle to show that showing contraction for pairs of configurations, where at least one of them is close enough to the equilibrium, is sufficient for establishing rapid mixing.

\medskip

Our main result is general enough to be applied to statistical mechanical models that undergo both types of phase transitions and to models whose macroscopic quantity are in higher dimensions.  Moreover, despite the generality, the application of our results requires straightforward conditions that we illustrate in Section \ref{GCWP}.  This is a significant simplification for proving rapid mixing for statistical mechanical models, especially those that undergo first-order, discontinuous phase transitions.  Lastly, our results also provide a link between measure concentration of the stationary distribution and rapid mixing of the corresponding dynamics for this class of statistical mechanical models.  This idea has been previously studied in \cite{MJLuczak} where the main result showed that rapid mixing implied measure concentration defined in terms of Lipschitz functions.  In our work, we prove a type of converse where measure concentration, in terms of a large deviation principle, implies rapid mixing.

\medskip

The paper is organized as follows. In Section \ref{sec:gibbs} the general construction of the class of the mean-field models considered in this paper is provided. Next, in Section \ref{ldp} the large deviation principle for the Gibbs measures from  \cite{EHT} that will be used in the main result of this manuscript is reviewed, and the concept of equilibrium macrostates is discussed.  In Section \ref{sec:glauber} and Section \ref{sec:glauberprob} the Glauber dynamics is introduced, and its transition probabilities are analyzed. Section \ref{glauber} provides a greedy coupling construction, standard for the Glauber dynamics of a mean-field statistical mechanical model. Section \ref{sec:mean} describes a single time step evolution of the mean coupling distance for two configurations whose spin proportion vectors are $\ve$-close. Section \ref{sec:mean} is followed by Section \ref{sec:agg} which describes a single time step evolution of the mean coupling distance in general by aggregating  the mean coupling distances along a monotone path of points connecting two configurations. Also, in  Section \ref{sec:agg} general conditions for the main result Theorem \ref{thm:main} are stated and discussed. The main result is stated and proved in Section \ref{sec:main}. The paper concludes with Section \ref{GCWP}, where the region of rapid mixing $\beta < \beta_s(q,r)$ for the generalized Curie-Weiss-Potts model (including the standard Curie-Weiss-Potts model) is proven as an immediate and simple application of the main result  of the current paper.

\bigskip

\section{Gibbs Ensembles} \label{sec:gibbs} 

\medskip

We begin by defining the general class of statistical mechanical spin models for which our results can be applied.  As mentioned above, this class includes all of the models listed in the introduction and we illustrate the application of our main result for the particular model: the Curie-Weiss-Potts model, in section \ref{GCWP}.  

\medskip

Let $q$ be a fixed integer and define $\Lambda = \{ e^1, e^2, \ldots, e^q \}$, where $e^k$ are the $q$ standard basis vectors of $\R^q$.  A {\it configuration} of the model has the form $\omega = (\omega_1, \omega_2, \ldots, \omega_n) \in \Lambda^n$.  We will consider a configuration on a graph with $n$ vertices and let $X_i(\omega) = \omega_i$ be the {\it spin} at vertex $i$.  The random variables $X_i$'s for $i=1, 2, \ldots, n$ are independent and identically distributed with common distribution $\rho$.  

\medskip

In terms of the microscopic quantities, the spins at each vertex, the relevant macroscopic quantity is the {\it magnetization vector} (aka empirical measure or proportion vector) 
\be
\label{eqn:empmeasure} 
L_n(\omega) = (L_{n,1}(\omega), L_{n,2}(\omega), \ldots, L_{n,q}(\omega)), 
\ee
where the $k$th component is defined by 
\[ L_{n, k}(\omega) = \frac{1}{n} \sum_{i=1}^n \delta(\omega_i, e^k) \]
which yields the proportion of spins in configuration $\omega$ that take on the value $e^k$.  The magnetization vector $L_n$ takes values in the set of probability vectors 
\be
\label{eqn:latticesimplex} 
\mathcal{P}_n = \left\{ \frac{n_k}{n} : \mbox{each} \ n_k \in \{0, 1, \ldots, n\} \ \mbox{and} \ \sum_{k=1}^q n_k = n \right\} 
\ee
inside the continuous simplex
$$ \mathcal{P} = \left\{ \nu \in \R^q : \nu = (\nu_1, \nu_2, \ldots, \nu_q), \mbox{each} \ \nu_k \geq 0, \ \sum_{k=1}^q \nu_k = 1 \right\}.$$

\medskip

\begin{remark}
For $q=2$, the empirical measure $L_n$ yields the empirical mean $S_n(\omega)/n$ where $S_n(\omega) = \sum_{i=1}^n \omega_i$.  Therefore, the class of models considered in this paper includes those where the relevant macroscopic quantity is the empirical mean, like the Curie-Weiss (mean-field Ising) model.
\end{remark}

Statistical mechanical models are defined in terms of the Hamiltonian function, which we denote by $H_n(\omega)$.  The Hamiltonian function encodes the interactions of the individual spins and the total energy of a configuration.  To take advantage of the large deviation bounds stated in the next section, we assume that the Hamiltonian can be expressed in terms of the empirical measures $L_n$ as stated in the following definition.  

\begin{defn}
\label{defn:irf}
For $z \in \R^q$, we define the {\bf interaction representation function}, denoted by $H(z)$, to be a differentiable function satisfying 
\[ H_n (\omega) = n H(L_n(\omega)) \] 
\end{defn}

\medskip

\noi
Throughout the paper we suppose the interaction representation function $H(z)$ is a finite concave $\mathcal{C}^3(\R^q)$ function that has the form
$$H(z)=H_1(z_1)+H_2(z_2)+\ldots +H_q(z_q)$$
For the Curie-Weiss-Potts (CWP) model discussed in section \ref{GCWP},
$$H(z)=-{1 \over 2} \big<z,z \big>=-{1 \over 2} z_1^2-{1 \over 2} z_2^2-\ldots -{1 \over 2} z_q^2.$$

\medskip

\begin{defn}
The {\bf Gibbs measure} or {\bf Gibbs ensemble} in statistical mechanics is defined as
\be 
\label{eqn:gibbs}
P_{n, \beta} (B) = \frac{1}{Z_n(\beta)} \int_B \exp \left\{ -\beta H_n(\omega) \right\} dP_n = \frac{1}{Z_n(\beta)} \int_B \exp \left\{  -\beta n \, H\left(L_n(\omega) \right) \right\} dP_n 
\ee
where $P_n$ is the product measure with identical marginals $\rho$ and $Z_n(\beta) = \int_{\Lambda^n} \exp \left\{ -\beta H_n(\omega) \right\} dP_n$ is the {\bf partition function}.  The positive parameter $\beta$ represents the inverse temperature of the external heat bath.
\end{defn}

\medskip

\begin{remark}
To simplify the presentation, we take $\Lambda = \{ e^1, e^2, \ldots, e^q \}$, where $e^k$ are the $q$ standard basis vectors of $\R^q$.  But our analysis has a straight-forward generalization to the case where $\Lambda = \{ \theta^1, \theta^2, \ldots, \theta^q \}$, where $\theta^k$ is any basis of $\R^q$.  In this case, the product measure $P_n$ would have identical one-dimensional marginals equal to 
\[ \bar{\rho} = \frac{1}{q} \sum_{i=1}^q \delta_{\theta^i} \]
\end{remark}

\medskip

An important tool we use to prove rapid mixing of the Glauber dynamics that converge to the Gibbs ensemble above is the large deviation principle of the empirical measure with respect to the Gibbs ensemble.  This measure concentration is precisely what drives the rapid mixing.  The large deviations background is presented next.

\bigskip

\section{Large Deviations} \label{ldp} 

\medskip

By Sanov's Theorem, the empirical measure $L_n$ satisfies the large deviation principle (LDP) with respect to the product measure $P_n$ with identical marginals $\rho$ and the rate function is given by the {\bf relative entropy}
\[ R(\nu | \rho) = \sum_{k=1}^q \nu_k \log \left( \frac{\nu_k}{\rho_k} \right) \]
for $\nu \in \mathcal{P}$.  Theorem 2.4 of \cite{EHT} yields the following result for the Gibbs measures (\ref{eqn:gibbs}).
\begin{thm}
The empirical measure $L_n$ satisfies the LDP with respect to the Gibbs measure $P_{n, \beta}$ with rate function
\[ I_\beta(z) = R(z | \rho) + \beta H(z) - \inf_t \{ R(t | \rho) + \beta H(t) \} \]
In other words, for any closed subset $F$,
\be
\label{eqn:ldpbound}
 \limsup_{n \goto \infty} \frac{1}{n} \log P_{n, \beta} ( L_n \in F ) \leq - \inf_{z \in F} I_\beta (z) 
 \ee
 and for any open subset $G$,
 \[ 
  \liminf_{n \goto \infty} \frac{1}{n} \log P_{n, \beta} ( L_n \in G ) \geq - \inf_{z \in G} I_\beta (z).
\]  
\end{thm}

\medskip

The LDP upper bound (\ref{eqn:ldpbound}) stated in the previous theorem yields the following natural definition of {\bf equilibrium macrostates} of the model.
\be
\label{eqn:eqmac} 
\mathcal{E}_\beta := \left\{ \nu \in \mathcal{P} : \nu \ \mbox{minimizes} \ R(\nu | \rho) + \beta H(\nu) \right\} 
\ee
For our main result, we assume that there exists a positive interval $B$ such that for all $\beta \in B$, $\mathcal{E}_\beta$ consists of a single state $z_\beta$.  We refer to this interval $B$ as the single phase region.

\medskip

Again from the LDP upper bound, when $\beta$ lies in the single phase region, we get
\be
\label{eqn:ldplimit}
P_{n, \beta} (L_n \in dx) \Longrightarrow \delta_{z_\beta} \hsp \mbox{as} \hsp n \goto \infty. 
\ee
The above asymptotic behavior will play a key role in obtaining a rapid mixing time rate for the Glauber dynamics corresponding to the Gibbs measures (\ref{eqn:gibbs}).

\medskip

An important function in our work is the free energy functional defined below.  It is defined in terms of the interaction representation function $H$ and the logarithmic moment generating function of a single spin; specifically, for $z \in \R^q$ and $\rho$ equal to the uniform distribution, the {\bf logarithmic moment generating function} of $X_1$, the spin at vertex $1$, is defined by 
\be 
\label{eqn:lmgf}
\Gamma(z) = \log \left(  \frac{1}{q} \sum_{k=1}^q \exp\{z_k\}\right). 
\ee

\begin{defn}
\label{defn:fef}
The {\bf free energy functional} for the Gibbs ensemble $P_{n, \beta}$ is defined as
\be 
\label{eqn:fef}
G_{\beta} (z) = \beta (-H)^\ast (-\nabla H(z)) - \Gamma( -\beta \nabla H(z))
\ee
where for a finite, differentiable, convex function $F$ on $\R^q$, $F^\ast$ denotes its Legendre-Fenchel transform defined by 
\[ F^\ast(z) = \sup_{x \in \R^q} \{ \langle x, z \rangle - F(x) \} \]
\end{defn}

The following lemma yields an alternative formulation of the set of equilibrium macrostates of the Gibbs ensemble in terms of the free energy functional.  The proof is a straightforward generalization of Theorem A.1 in \cite{CET}.  

\begin{lemma}
\label{lemma:fefeqstates}
Suppose $H$ is finite, differentiable, and concave.  Then 
\[ \inf_{z \in \mathcal{P}} \{ R(z | \rho) + \beta H(z) \} = \inf_{z \in \R^q}  \{ G_\beta(z) \} \]
Moreover, $z_0 \in \mathcal{P}$ is a minimizer of $R(z | \rho) + \beta H(z)$ if and only if $z_0$ is a minimizer of $G_\beta(z)$.
\end{lemma}

\noi 
Therefore, the set of equilibrium macrostates can be expressed in terms of the free energy functional as 
\be
\label{eqn:fefeqst}
\mathcal{E}_\beta = \left\{ z \in \mathcal{P} : z \ \mbox{minimizes} \ G_\beta(z) \right\} 
\ee

As mentioned above, we consider only the single phase region of the Gibbs ensemble; i.e.\ values of $\beta$ where $G_\beta(z)$ has a unique global minimum.  For example, for the Curie-Weiss-Potts model, the single phase region are values of $\beta$ such that $0 < \beta < \beta_c := (2(q-1)/(q-2)) \log(q-1)$.  At this critical value $\beta_c$, the model undergoes a first-order, discontinuous phase transition in which the single phase changes to a multiple phase discontinuously.  This is discussed in detail in Section \ref{GCWP}.

\medskip

As we will show, the geometry of the free energy functional $G_\beta$ not only determines the equilibrium behavior of the Gibbs ensembles but it also yields the condition for rapid mixing of the corresponding Glauber dynamics.

\bigskip

\section{Glauber Dynamics and Mixing Times} \label{sec:glauber} 

\medskip

On the configuration space $\Lambda^n$, we define the Glauber dynamics for the class of spin models considered in this paper.  These dynamics yield a reversible Markov chain $X^t$ with stationary distribution being the Gibbs ensemble $P_{n, \beta}$.

\medskip

(i) Select a vertex $i$ uniformly, 

\smallskip

(ii) Update the spin at vertex $i$ according to the distribution $P_{n, \beta}$, conditioned on the event that the spins at all vertices not equal to $i$ remain unchanged.  

\medskip

For a given configuration $\sigma = (\sigma_1, \sigma_2, \ldots, \sigma_n)$, denote by $\sigma_{i, e^k}$ the configuration that agrees with $\sigma$ at all vertices $j \neq i$ and the spin at the vertex $i$ is $e^k$; i.e. 
\[ \sigma_{i, e^k} = (\sigma_1, \sigma_2, \ldots, \sigma_{i-1}, e^k, \sigma_{i+1}, \ldots, \sigma_n) \]

Then if the current configuration is $\sigma$ and vertex $i$ is selected, the probability the spin at $i$ is updated to $e^k$, denoted by $P(\sigma \goto \sigma_{i, e^k})$, is equal to 
\be
\label{eqn:tranprob1} 
P(\sigma \goto \sigma_{i, e^k}) = \frac{\exp\big\{-\beta n H(L_n(\sigma_{i, e^k}))\big\}}{ \sum_{\ell=1}^q \exp\big\{-\beta n H(L_n(\sigma_{i, e^\ell}))\big\}}. 
\ee

\medskip

The mixing time is a measure of the convergence rate of a Markov chain to its stationary distribution and is defined in terms of 
the {\bf total variation distance} between two distributions $\mu$ and $\nu$ on the configuration space $\Omega$ is defined by
\[ \| \mu - \nu \|_{\TV} = \sup_{A \subset \Omega} | \mu(A) - \nu(A)| = \frac{1}{2} \sum_{x \in \Omega} | \mu(x) - \nu(x)| \]
Given the convergence of the Markov chain, we define the maximal distance to stationary to be 
\[ d(t) = \max_{x \in \Omega} \| P^t(x, \cdot) - \pi \|_{\TV} \]
where $P^t(x, \cdot)$ is the transition probability of the Markov chain starting in configuration $x$ and $\pi$ is its stationary distribution.
Then, given $\epsilon>0$, the {\bf mixing time} of the Markov chain is defined by
\[ t_{mix}(\epsilon) = \min\{ t : d(t) \leq \epsilon \} \smallskip \]
See \cite{LPW} for a detailed survey on the theory of mixing times. 

\medskip

Rates of mixing times are generally categorized into two groups: {\it rapid mixing} which implies that the mixing time exhibits polynomial growth with respect to the system size $n$, and {\it slow mixing} which implies that the mixing time grows exponentially with the system size.  Determining the parameter regime where a model undergoes rapid mixing is of major importance, as it is in this region that the application of the dynamics is physically feasible.  

\bigskip

\section{Glauber Dynamics Transition Probabilities} \label{sec:glauberprob} 

\medskip

In this section, we show that the update probabilities of the Glauber dynamics introduced in the previous section can be expressed in terms of the derivative of the logarithmic moment generating function of the individual spins $\Gamma$ defined in (\ref{eqn:lmgf}). The partial derivative of $\Gamma$ in the direction of $e^\ell$ has the form 
\[ \left[\partial_{\ell} \Gamma \right] (z) = \frac{\exp\{z_\ell\}}{\sum_{k=1}^q \exp\{z_k\}} \]

\medskip

\noi
We introduce the following function that plays the key role in our analysis.
\be
\label{eqn:gell}
g_{\ell}^{H, \beta}(z) = \left[ \partial_{\ell} \Gamma \right] (-\beta \nabla H(z)) = \frac{\exp\left( -\beta \, [\partial_\ell H](z) \right)}{ \sum_{k=1}^q \exp \left( -\beta \, [\partial_k H](z)  \right) }.
\ee
Denote
\be
\label{eqn:littleG}
g^{H, \beta}(z):=\Big(g_1^{H, \beta}(z),\ldots,g_q^{H, \beta}(z) \Big).
\ee
Note that $g^{H, \beta}(z)$ maps the simplex
$$ \mathcal{P} = \left\{ \nu \in \R^q : \nu = (\nu_1, \nu_2, \ldots, \nu_q), \mbox{each} \ \nu_k \geq 0, \ \sum_{k=1}^q \nu_k = 1 \right\} $$
into itself and it can be expressed in terms of the free energy functional $G_\beta$ defined in (\ref{eqn:fef}) by 
\[ \nabla G_\beta(z) = \beta [ \nabla (-H)^\ast (-\nabla H(z)) - g^{H, \beta}(z) ] \]

\medskip

\begin{lemma}
\label{lemma:transprob}
Let $P(\sigma \goto \sigma_{i, e^k})$ be the Glauber dynamics update probabilities given in (\ref{eqn:tranprob1}).  
Then, for any $k \in \{1,2,\ldots, q \}$,
$$P(\sigma \goto \sigma_{i, e^k}) =  \left[ \partial_k \Gamma \right] \Big(-\beta \nabla H(L_n(\sigma))-{\beta \over 2n}\mathcal{Q}H(L_n(\sigma))+{\beta \over n} \Big< \sigma_i, \mathcal{Q}H(L_n(\sigma))\Big> \sigma_i \Big) + O\left(\frac{1}{n^2} \right),$$
where $\mathcal{Q}$ is the following linear operator:
$$\mathcal{Q} F(z):=\left(\partial_1^2 F(z), ~\partial_2^2 F(z), ~\ldots , ~\partial_q^2 F(z)\right),$$
for any $~F: \mathbb{R}^q \rightarrow \mathbb{R}~$ in $\mathcal{C}^2$.
\end{lemma}
\bp
Suppose $\sigma_i=e^m$. By Taylor's theorem, for any $k \not= m$, we have 
\beas
H(L_n(\sigma_{i, e^k})) & = & H(L_n(\sigma)) +H_m\big(L_{n, m} (\sigma) -1/n \big)-H_m\big(L_{n, m} (\sigma) \big) \\
&  & + ~H_k\big(L_{n, k} (\sigma) +1/n \big)-H_k\big(L_{n, k} (\sigma) \big) \\
& = & H(L_n(\sigma)) + \frac{1}{n} \left[ \partial_{k} H(L_n(\sigma)) - \partial_m H(L_n(\sigma)) \right]\\
& & +~{1 \over 2n^2} \left[ \partial^2_{k} H(L_n(\sigma)) + \partial^2_m H(L_n(\sigma)) \right]+ O\left(\frac{1}{n^3} \right). 
\eeas
Now, if $k=m$,
\beas
H(L_n(\sigma_{i, e^k}))  & = & H(L_n(\sigma))\\
& = & H(L_n(\sigma)) + \frac{1}{n} \left[ \partial_{k} H(L_n(\sigma)) -\partial_m H(L_n(\sigma)) \right] \\ 
& & \hsp +{1 \over 2n^2} \left[ -\partial^2_{k} H(L_n(\sigma)) +\partial^2_m H(L_n(\sigma)) \right]. 
\eeas
This implies that the transition probability (\ref{eqn:tranprob1}) has the form
\[
P(\sigma \goto \sigma_{i, e^k}) = \left[ \partial_k \Gamma \right] \Big(-\beta \nabla H(L_n(\sigma))-{\beta \over 2n}\mathcal{Q}H(L_n(\sigma))+{\beta \over n}\partial^2_{m}H(L_n(\sigma))e^m \Big) + O\left(\frac{1}{n^2} \right) 
\]
as $\exp\big\{O\left(\frac{1}{n^2} \right) \big\}=1+O\left(\frac{1}{n^2} \right)$.
\ep

\medskip

\noindent
The above Lemma \ref{lemma:transprob} can be restated as follows using Taylor expansion.
\begin{cor} \label{cor:tp}
Let $P(\sigma \goto \sigma_{i, e^k})$ be the Glauber dynamics update probabilities given in (\ref{eqn:tranprob1}).  
Then, for any $k \in \{1,2,\ldots, q \}$,
$$P(\sigma \goto \sigma_{i, e^k}) = g_k^{H, \beta}(L_n(\sigma)) +{ \beta \over n} \varphi_{k, \sigma_i}^{H, \beta}(L_n(\sigma))+O\left(\frac{1}{n^2} \right),$$
where
$$\varphi_{k, e^r}^{H, \beta}(z):=-{1 \over 2}\Big<\mathcal{Q}H(z),  \left[\nabla \partial_{k} \Gamma \right] (-\beta \nabla H(z)) \Big>+\Big< e^r, \mathcal{Q} H(z)\Big> \Big<e^r,  \left[\nabla \partial_{k} \Gamma \right] (-\beta \nabla H(z)) \Big>  .$$
\end{cor}

\medskip

As indicated in the title of the paper, we employ a coupling method for proving rapid mixing of Glauber dynamics of Gibbs ensembles.  In the next section, we define the specific coupling used.
 
\bigskip

\section{Coupling of Glauber Dynamics}
\label{glauber}

\medskip

We begin by defining a metric on the configuration space $\Lambda^n$.  For two configurations $\sigma$ and $\tau$ in $\Lambda^n$, define
\be
\label{eqn:pathmetric} 
d(\sigma, \tau) = \sum_{j=1}^n 1\{ \sigma_j \neq \tau_j \} 
\ee
which yields the number of vertices at which the two configurations differ.

\medskip

Let $X^t$ and $Y^t$ be two copies of the Glauber dynamics. Here, we use the standard greedy coupling of $X^t$ and $Y^t$.  At each time step a vertex is selected at random, uniformly from the $n$ vertices.     Suppose $X^t=\sigma$, $~Y^t=\tau$, and the vertex selected is denoted by $j$. Next, we erase the spin at location $j$ in both processes, and replace it with a new one according to the following update probabilities. For all $\ell = 1, 2, \ldots, q$, define
\[ p_\ell = P(\sigma \goto \sigma_{j, e^\ell}) \quad \mbox{and} \quad q_\ell = P(\tau \goto \tau_{j, e^\ell}) \]
and let
\[ P_\ell = \min\{ p_\ell, q_\ell \} \hsp \mbox{and} \hsp P = \sum_{\ell = 1}^q P_\ell . \]
Now, let $B$ be a Bernoulli random variable with probability of success $P$.  If $B = 1$, we update the two chains equally with the following probabilities
\[ P(X_j^{t+1} = e^\ell, Y_j^{t+1} = e^\ell \, | \, B=1) = \frac{P_\ell}{P} \]
for $\ell = 1, 2, \ldots, q$.  On the other hand, if $B= 0$, we update the chains differently according to the following probabilities
\[ P(X_j^{t+1} = e^\ell, Y_j^{t+1} = e^m \, | \, B=0) = \frac{p_\ell - P_\ell}{1-P} \cdot \frac{q_m - P_m}{1-P} \] 
for all pairs $\ell \neq m$.  Then the total probability that the two chains update the same is equal to $P$ and
the total probability that the chains update differently is equal to $1-P$.

\medskip

\noi
Observe that once $X^t=Y^t$, the processes remain matched (coupled) for the rest of the time. In the coupling literature, the time 
$$\min\{t \geq 0 ~:~ X^t=Y^t\}$$
is refered to as the {\it coupling time}.

\medskip

The mean coupling distance  $~\mathbb{E}[d(X^t, Y^t) ]~$ is tied to the total variation distance via the following inequality known as  the {\it coupling inequality:} 
\be \label{coupling_ineq} 
\| P^t(x, \cdot) - P^t(y, \cdot) \|_{{\scriptsize \mbox{TV}}} \leq P(X^t \neq Y^t)  \leq \mathbb{E}[d(X^t, Y^t) ] 
\ee
The above inequality implies that the order of the mean coupling time is an upper bound on the order of the mixing time. See \cite{L} and \cite{LPW} for details on coupling and coupling inequalities.

\medskip

\section{Mean Coupling Distance} \label{sec:mean} 

\medskip
Fix $~\varepsilon>0$. Consider two configurations $\sigma$ and $\tau$ such that
$$d(\sigma,\tau)=d,$$
where $d(\sigma, \tau) \in \mathbb{N}$ is the metric defined in (\ref{eqn:pathmetric}) and $\ve \leq \|L_n(\sigma)-L_n(\tau)\|_1 < 2\varepsilon$.

\medskip
Let $\mathcal{I}=\{i_1,\hdots,i_d\}$ be the set of vertices at which the spin values of the two configurations $\sigma$ and $\tau$ disagree.  Define $\kappa(e^\ell)$ to be the probability that the coupled processes update differently when the chosen vertex $j\not\in \mathcal{I}$ has spin $e^\ell$. If the chosen vertex $j$ is such that $\sigma_j=\tau_j=e^\ell$, then expressing $\kappa(e^\ell)$ by total variation distance and by Corollary \ref{cor:tp} of  Lemma \ref{lemma:transprob},

\bea\label{eqn:true} \nonumber
 \kappa(e^\ell) & :=  & {1 \over 2}\sum_{k=1}^q \Big|P(\sigma \goto \sigma_{j, e^k}) -P(\tau \goto \tau_{j, e^k}) \Big|\\ \nonumber
 & & \\ \nonumber
& =  & {1 \over 2}\sum_{k=1}^q \Big|  \Big(g_k^{H, \beta}(L_n(\sigma)) +{ \beta \over n} \varphi_{k, e^\ell}^{H, \beta}(L_n(\sigma)) \Big)
- \Big(g_k^{H, \beta}(L_n(\tau)) +{ \beta \over n} \varphi_{k, e^\ell}^{H, \beta}(L_n(\tau)) \Big)  \Big| +O\left( \frac{1}{n^2} \right)\\ 
 & & \\ \nonumber
& =  & {1 \over 2}\sum_{k=1}^q \Big|  g_k^{H, \beta}(L_n(\sigma)) - g_k^{H, \beta}(L_n(\tau))   \Big| +O\left(\frac{\varepsilon}{n}+ \frac{1}{n^2} \right). \nonumber
\eea

\medskip
\noindent
Next, we observe that for any $\mathcal{C}^2$ function $f:~\mathcal{P} \rightarrow \mathbb{R}$, there exists $C>0$ such that
\be
\left| f(z') - f(z) - \Big<z' - z, \nabla f(z) \Big>~\right| < C\varepsilon^2
\ee
for all $z, z' \in \mathcal{P}$ satisfying $\varepsilon \leq \|z'-z\|_1 < 2\varepsilon$.

\bigskip
\noindent
Therefore for $n$ large enough, there exists $C'>0$ such that
\begin{equation}\label{eqn:rho}
\left| \kappa(e^\ell) ~-~ {1 \over 2}\sum_{k=1}^q \Big| \Big<L_n(\tau) - L_n(\sigma), \nabla g_k^{H, \beta}(L_n(\sigma)) \Big> \Big| \right| ~<C'\varepsilon^2.
\end{equation}
The above result holds regardless of the value of $\ell \in \{1,2,\dots,q\}$. 

\medskip
\noindent
Similarly, when the chosen vertex $j \in \mathcal{I}$, the probability of not coupling at $j$ satisfies (\ref{eqn:rho}).

\bigskip
\noindent
We conclude that in terms of $\kappa_{\sigma,\tau} := {1 \over 2}\sum_{k=1}^q \Big| \Big<L_n(\tau) - L_n(\sigma), \nabla g_k^{H, \beta}(L_n(\sigma)) \Big> \Big|$, the mean distance between a coupling of the Glauber dynamics starting in $\sigma$ and $\tau$ with $d(\sigma, \tau) = d$ after one step has the form
\begin{eqnarray} \label{eqn:grad}
\mathbb{E}_{\sigma,\tau}[d(X,Y)] & \leq  & d-{d \over n} (1-\kappa_{\sigma,\tau})+{n-d \over n}\kappa_{\sigma,\tau} +c \, \varepsilon^2 \nonumber \\
&  = & d \cdot \left[1-{1 \over n}\left(1-{\kappa_{\sigma,\tau}+c \, \varepsilon^2  \over d/n} \right)  \right]
\end{eqnarray} 
for a fixed $c>0$ and all $\ve$ small enough.

\medskip

\section{Aggregate Path Coupling} \label{sec:agg}

\medskip

In the previous section, we derived the form of the mean distance between a coupling of the Glauber dynamics starting in two configurations whose distance is bounded.  We next derive the form of the mean coupling distance of a coupling starting in two configurations that are connected by a path of configurations where the distance between successive configurations are bounded.   

\begin{defn}
Let $\sigma$ and $\tau$ be configurations in $\Lambda^n$. We say that a path $\pi$ connecting configurations $\sigma$ and $\tau$ denoted by 
$$\pi: \ \sigma = x_0, x_1, \ldots, x_r = \tau ,$$
is a {\bf monotone path} if
\begin{itemize}
\item[(i)] $~\sum\limits_{i=1}^r d(x_{i-1},x_i) = d(\sigma,\tau)$

\item[(ii)] for each $k=1, 2, \ldots, q$, the $k$th coordinate of $L_n(x_i)$, $L_{n,k}(x_i)$ is monotonic as $i$ increases from $0$ to $r$;\\
\end{itemize}
\end{defn}
\noi
Observe that here the points $x_i$ on the path are not required to be nearest-neighbors.

\medskip

\noi
A straightforward property of monotone paths is that 
\[ \sum_{i=1}^r \sum_{k=1}^q | L_{n,k}(x_i) - L_{n,k}(x_{i-1}) | = \| L_n(\sigma) - L_n(\tau) \|_1   \]
Another straightforward observation is that for any given path
$$L_n(\sigma)=z_0, z_1,\hdots,z_r=L_n(\tau)$$
in $\mathcal{P}_n$, monotone in each coordinate, with $~\|z_i-z_{i-1}\|_1>0$ for all $i \in \{1,2,\hdots,r\}$, there exists a monotone path
$$\pi: \ \sigma = x_0, x_1, \ldots, x_r = \tau$$
such that $L_n(x_i)=z_i$ for each $i$.

\medskip

Let $\pi : \sigma = x_0, x_1, \hdots, x_r = \tau$ \ be a monotone path connecting configurations $\sigma$ and $\tau$ such that $\varepsilon \leq \|L_n(x_i) - L_n(x_{i-1}) \|_1 < 2\varepsilon $ for all $i = 1, \ldots, r$.   Equation (\ref{eqn:grad}) implies the following bound on the mean distance between a coupling of the Glauber dynamics starting in configurations $\sigma$ and $\tau$:

{\small
\bea \label{eqn:meandist}
\lefteqn{ \mathbb{E}_{\sigma, \tau}[d(X,Y)] } \nonumber \\
& \leq & \sum_{i=1}^r \mathbb{E}_{x_{i-1}, x_i}[d(X_{i-1}, X_i)]   \nonumber \\
 & \leq &  \sum\limits_{i=1}^r \left\{ d(x_{i-1},x_i) \cdot \left[1 -{1 \over n} \left(1-{{1 \over 2} \sum\limits_{k=1}^q \Big| \Big<L_n(x_i) - L_n(x_{i-1}), \nabla g_k^{H, \beta}(L_n( x_{i-1})) \Big> \Big|+c \varepsilon^2 \over d(x_{i-1},x_i)/n} \right)  \right]\right\}  \nonumber \\ 
& = & d(\sigma,\tau) \left[ 1-{1 \over n}\left(1-{\sum\limits_{k=1}^q  \sum\limits_{i=1}^r \Big| \Big<L_n(x_i) - L_n(x_{i-1}), \nabla g_k^{H, \beta}(L_n( x_{i-1})) \Big> \Big|+c \varepsilon^2 \over  2d(\sigma,\tau)/n }\right) \right]  \nonumber \\ 
 & & \\
& \leq & d(\sigma,\tau) \left[ 1-{1 \over n}\left(1-{\sum\limits_{k=1}^q  \sum\limits_{i=1}^r \Big| \Big<L_n(x_i) - L_n(x_{i-1}), \nabla g_k^{H, \beta}(L_n( x_{i-1})) \Big> \Big|+c \varepsilon^2 \over  \|L_n(\sigma)-L_n(\tau)\|_1 }\right) \right], \nonumber
\eea
}

\noindent
as $~\sum\limits_{i=1}^r d(x_{i-1},x_i) = d(\sigma,\tau)$.

\medskip

From inequality (\ref{eqn:meandist}), if there exists monotone paths between all pairs of configurations such that there is a uniform bound less than $1$ on the ratio 
\[ {\sum_{k=1}^q  \sum_{i=1}^r \Big| \Big<L_n(x_i) - L_n(x_{i-1}), \nabla g_k^{H, \beta}(L_n( x_{i-1})) \Big> \Big| \over  \|L_n(\sigma)-L_n(\tau)\|_1 } \]
then the mean coupling distance contracts which yields a bound on the mixing time via coupling inequality (\ref{coupling_ineq}).

\medskip

Although the Gibbs measure are distributions of the empirical measure $L_n$ defined on the discrete space $\mathcal{P}_n$, proving contraction of the mean coupling distance is often facilitated by working in the continuous space, namely the simplex $\mathcal{P}$.  We begin our discussion of aggregate path coupling by defining distances along paths in $\mathcal{P}$.  

\medskip

Recall the function $g^{H,\beta}$ defined in (\ref{eqn:littleG}) which is dependent on the Hamiltonian of the model through the interaction representation function $H$ defined in Definition \ref{defn:irf}. 

\begin{defn}
Define the {\bf aggregate $g$-variation} between a pair of points $x$ and $z$ in $\mathcal{P}$ along a continuous monotone  (in each coordinate) path $\rho$ to be 
$$ D_\rho^g (x, z) := \sum\limits_{k=1}^q\int\limits_{\rho} \Big| \Big<\nabla g_k^{H, \beta}(y), dy \Big> \Big| $$
Define the corresponding {\bf pseudo-distance} between a pair of points points $x$ and $z$ in $\mathcal{P}$ as
$$d_g(x,z) :=\inf_{\rho}D_\rho^g (x, z),$$
where the infimum is taken over all continuous monotone paths in $\mathcal{P}$ connecting $x$ and $z$. 
\end{defn}

Notice if the monotonicity restriction is removed, the above infimum would satisfy the triangle inequality.
We will need the following condition.
\begin{cond} \label{uniform}
Let $z_\beta$ be the unique equilibrium macrostate. There exists $\delta \in (0,1)$ such that
$${d_g(z,z_\beta) \over \|z-z_\beta\|_1} \leq 1-\delta$$
for all $z$ in $\mathcal{P}$.
\end{cond}
\noindent
Observe that  if it is shown that $~d_g(z,z_\beta) < \|z-z_\beta\|_1$ for all $z$ in $\mathcal{P}$, then by continuity the above condition is equivalent to
$$\limsup_{z \rightarrow z_\beta} {d_g(z,z_\beta) \over \|z-z_\beta\|_1} <1$$

\medskip
\noindent
Suppose Condition \ref{uniform} is satisfied. Then let denote by ${\bf NG}_\delta$ the family of {\bf neo-geodesic} smooth curves, monotone in each coordinate such that for each $z \not= z_\beta$ in $\mathcal{P}$, there is exactly one curve $\rho=\rho_z$ in the family  ${\bf NG}_\delta$ connecting $z_\beta$ to $z$, and
$${D_\rho^g (z, z_\beta) \over \|z-z_\beta\|_1} \leq 1-\delta/2$$

\medskip
\noindent
\begin{cond}\label{RS}
For $\varepsilon>0$ small enough, there exists a neo-geodesic family ${\bf NG}_\delta$ such that for each $z$ in $\mathcal{P}$ satisfying $\|z-z_\beta\|_1 \geq \varepsilon$ , the curve  $\rho=\rho_z$ in the family  ${\bf NG}_\delta$ that connects $z_\beta$ to $z$ satisfies
$${\sum_{k=1}^q  \sum_{i=1}^r \Big| \Big<z_i - z_{i-1}, \nabla g_k^{H, \beta}(z_{i-1}) \Big> \Big| \over \|z-z_\beta\|_1} \leq 1-\delta/3$$
for a sequence of points $z_0=z_\beta,z_1,\hdots,z_r=z$ interpolating $\rho$ such that
$$\varepsilon \leq \|z_i - z_{i-1}\|_1  < 2\varepsilon \quad \text{ for } i=1,2,\hdots, r.$$
\end{cond}

\medskip
\noindent
It is important to observe that Condition \ref{uniform} is often simpler to verify than Condition \ref{RS}. Moreover,  under certain simple additional prerequisites, Condition \ref{uniform} implies Condition \ref{RS}. For example, this is achieved if there is a uniform bound on the Cauchy curvature at every point of every curve in ${\bf NG}_\delta$. So it will be demonstrated on the example of Curie-Weiss-Potts model that  the natural way for  establishing  Condition \ref{RS} for the model is via first establishing Condition \ref{uniform}.

\medskip

In addition to Condition \ref{RS} that will be shown to imply contraction when one of the two configurations in the coupled processes is at the equilibrium, i.e. $L_n(\sigma)=z_\beta$ , we need a condition that will imply contraction between two configurations within a neighborhood of the equilibrium configuration.  We state this assumption next.  

\begin{cond}
\label{ass:localbound}
Let $z_\beta$ be the unique equilibrium macrostate. Then,
$$\limsup\limits_{z \rightarrow z_\beta} {\|g^{H, \beta}(z)-g^{H, \beta}(z_\beta)\|_1 \over \|z-z_\beta\|_1} <1.$$
\end{cond}

\noindent
Since $H(z) \in \mathcal{C}^3$, the above Condition \ref{ass:localbound} implies that for any $\ve>0$ sufficiently small, there exists $\gamma \in (0,1)$ such that
$$ {\|g^{H, \beta}(z)-g^{H, \beta}(w)\|_1 \over \|z-w\|_1} <1-\gamma$$
for all $z$ and $w$ in $\mathcal{P}$ satisfying
$$\|z-z_\beta\|_1<\ve \quad \text{ and }\quad \|w-z_\beta\|_1<\ve.$$

\bigskip

\section{Main Result} \label{sec:main} 

\medskip

A sufficient condition for rapid mixing of the Glauber dynamics of Gibbs ensembles is contraction of the mean coupling distance $\mathbb{E}_{\sigma, \tau}[d(X,Y)]$ between coupled processes starting in all pairs of configurations in $\Lambda^n$.  The classical path coupling argument \cite{BD} is a method of obtaining this contraction by only proving contraction between couplings starting in \underline{neighboring} configurations.  However for some classes of models (e.g. models that undergo a first-order, discontinuous phase transition) there are situations when Glauber dynamics exhibits rapid mixing, but coupled processes do not exhibit contraction between some neighboring configurations. Such models include the mean-field Blume-Capel (in the discontinuous phase transition region) and Curie-Weiss-Potts models. A major strength of the {\bf aggregate path coupling} method introduced in \cite{KOT} for mean-field Blume-Capel model and further expanded in this study is that, in addition to its generality, it yields a proof for rapid mixing even in those cases when contraction of the mean distance between couplings starting in all pairs of neighboring configurations does not hold.

\medskip

The strategy is to take advantage of the large deviations estimates discussed in section \ref{ldp}.  Recall from that section that we assume that the set of equilibrium macrostates $\mathcal{E}_\beta$, which can be expressed in the form given in (\ref{eqn:fefeqst}), consists of a single point $z_\beta$.  Define an {\it equilibrium configuration} $\sigma_\beta$ to be a configuration such that 
\[ L_n(\sigma_\beta) = z_\beta = ((z_\beta)_1, (z_\beta)_2, \ldots, (z_\beta)_q). \]  
First we observe that in order to use the coupling inequality (\ref{coupling_ineq}) we need to show contraction of the mean coupling distance $\mathbb{E}_{\sigma, \tau}[d(X,Y)]$ between a  Markov chain initially distributed according to the stationary probability distribution $ P_{n, \beta}$ and a Markov chain starting at any given configuration.
Using large deviations we know that with high probability the former process starts near the equilibrium and stays near the equilibrium for long duration of time.

Our main result Theorem \ref{thm:main}  states that once we establish contraction of the mean coupling distance between two copies of a Markov chain where one of the coupled dynamics starts near an equilibrium configuration in Lemma \ref{lemma:contract}, then this contraction, along with the large deviations estimates of the empirical measure $L_n$, yields rapid mixing of the Glauber dynamics converging to the Gibbs measure. 

\medskip

Now, the classical path coupling relies on showing contraction along any monotone path connecting two configurations, in one time step. Here we observe that we only need to show contraction along \underline{one} monotone path connecting two configurations in order to have the mean coupling distance $\mathbb{E}_{\sigma, \tau}[d(X,Y)]$ contract in a single time step. However, finding even one monotone path with which we can show contraction in the equation (\ref{eqn:meandist}) is not easy. The answer to this is in finding a monotone path $\rho$ in $\mathcal{P}$ connecting the $L_n$ values of the two configurations, $\sigma$ and $\tau$, such that 
$${\sum\limits_{k=1}^q\int\limits_{\rho} \Big| \Big<\nabla g_k^{H, \beta}(y), dy \Big> \Big| \over  \|L_n(\sigma)-L_n(\tau)\|_1 } ~<1 $$
Although $\rho$ is a continuous  path in continuous space $\mathcal{P}$, it serves as Ariadne's thread for finding a monotone path
$$\pi:~\sigma=x_0,~x_1,\hdots,x_r=\tau$$
such that $L_n(x_0),~L_n(x_1),\hdots,L_n(x_r)$ in $\mathcal{P}_n$ are positioned along $\rho$, and
$$\sum\limits_{k=1}^q  \sum\limits_{i=1}^r \Big| \Big<L_n(x_i) - L_n(x_{i-1}), \nabla g_k^{H, \beta}(L_n( x_{i-1})) \Big> \Big|$$
is a Riemann sum approximating $~\sum\limits_{k=1}^q\int\limits_{\rho} \Big| \Big<\nabla g_k^{H, \beta}(y), dy \Big> \Big|$. Therefore we obtain
$${\sum\limits_{k=1}^q  \sum\limits_{i=1}^r \Big| \Big<L_n(x_i) - L_n(x_{i-1}), \nabla g_k^{H, \beta}(L_n( x_{i-1})) \Big> \Big| \over  \|L_n(\sigma)-L_n(\tau)\|_1 } ~<1,$$
that in turn implies contraction in (\ref{eqn:meandist}) for $\ve$ small enough and $n$ large enough. See Figure \ref{fig:simplex}.

\begin{figure} 
\includegraphics[scale=0.4]{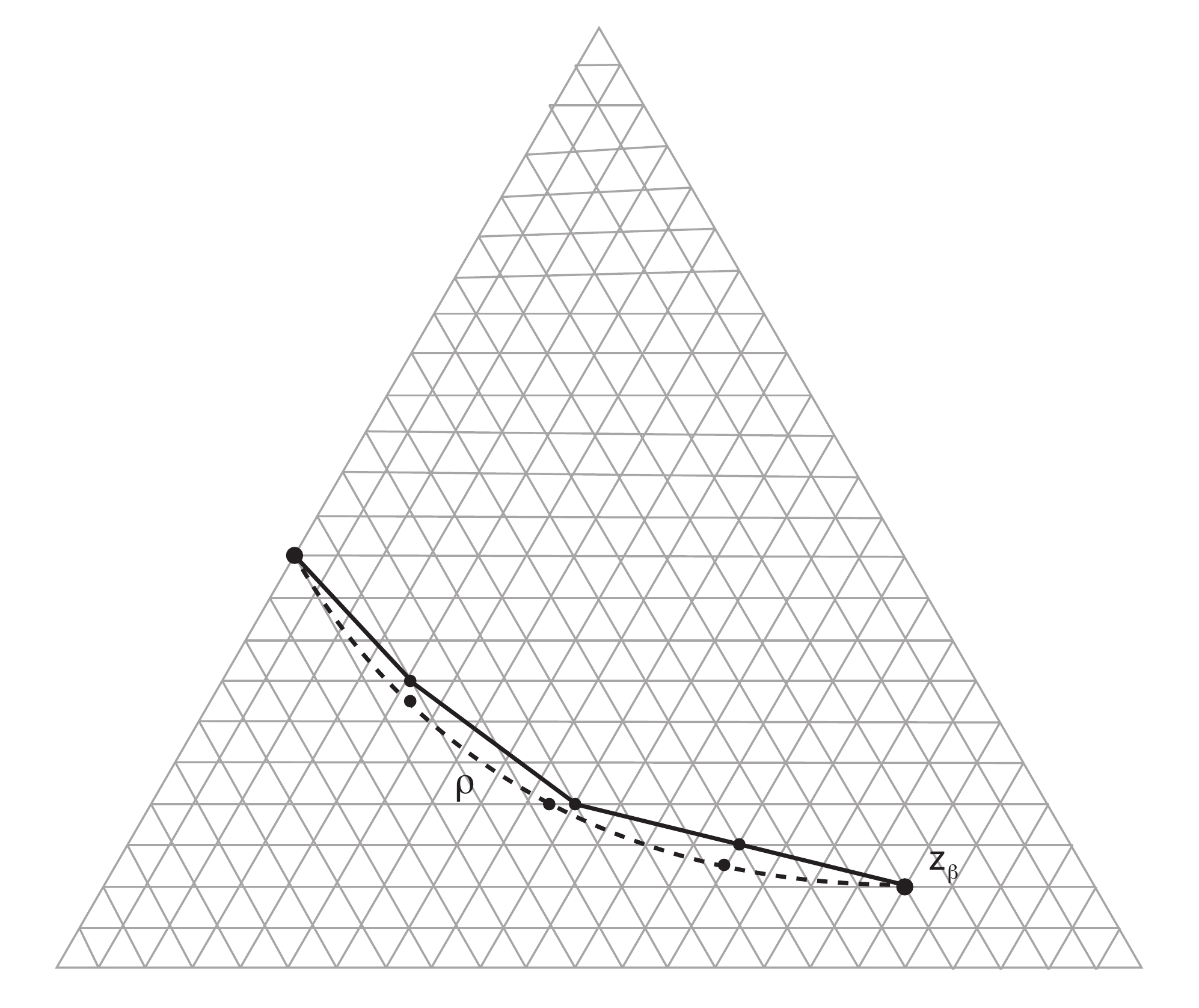}
\caption{Case $q=3$. Dashed curve is the continuous monotone path $\rho$. Solid lines represent the path $L_n(x_0),~L_n(x_1),\hdots,L_n(x_r)$ in $\mathcal{P}_n$.}
\label{fig:simplex}
\end{figure}

Observe that in order for the above argument to work, we need to spread points $L_n(x_i) \in \mathcal{P}_n$ along a continuous path $\rho$ at intervals of fixed order $\ve$.  Thus $\pi$ has to be {\bf not a nearest-neighbor path} in the space of configurations, another significant deviation from the classical path coupling.

\medskip

\begin{lemma}
\label{lemma:contract}
Assume Condition \ref{RS} and Condition \ref{ass:localbound}. Let $(X,Y)$ be a coupling of the Glauber dynamics as defined in Section \ref{glauber}, starting in configurations $\sigma$ and $\tau$ and let $z_\beta$ be the single equilibrium macrostate of the corresponding Gibbs ensemble. 
Then there exists an $\alpha>0$ and an $\varepsilon'>0$ small enough such that for $n$ large enough,
\[ \mathbb{E}_{\sigma,\tau} [ d(X,Y)] \leq e^{-\alpha/n} d(\sigma, \tau) \]
whenever $\|L_n(\sigma) -z_\beta \|_1<\varepsilon' $.
\end{lemma}
\bp
Let $\ve$ and $\delta$ be as in Condition \ref{RS}, and let $\ve'=\ve^2 \delta/M$ with a constant $M\gg0$. 

\medskip
\noindent
{\bf Case I.} Suppose $L_n(\tau)=z$ and $L_n(\sigma)=w$, where $\|z-z_\beta\|_1 \geq \ve$ and $\|w-z_\beta\|_1< \ve'$.

\medskip
\noindent
Then there is an equlibrium configuration $\sigma_\beta$ with $L_n(\sigma_\beta)=z_\beta$ such that there is a {\it monotone path}
$$\pi':~\sigma_\beta=x'_0,~x'_1,\hdots,x'_r=\tau$$
connecting configurations $\sigma_\beta$ and $\tau$ on $\Lambda^n$ such that
$\ve  \leq\|L_n(x'_i) - L_n(x'_{i-1})\|_1<2\ve$, and by Condition \ref{RS}, 
$${\sum\limits_{k=1}^q  \sum\limits_{i=1}^r \Big| \Big<L_n(x'_i) - L_n(x'_{i-1}), \nabla g_k^{H, \beta}(L_n( x'_{i-1})) \Big> \Big| \over  \|L_n(\sigma_\beta)-L_n(\tau)\|_1 } \leq 1-\delta/4$$
for $n$ large enough. Note that the difference between the above inequality and Condition \ref{RS} is that here we take $L_n(x'_i) \in \mathcal{P}_n$.

\medskip
\noindent
Now,  there exists a monotone path with from $\sigma$ to $\tau$
$$\pi:~\sigma=x_0,~x_1,\hdots,x_r=\tau$$
such that 
$$\|L_n(x_i)-L_n(x'_i)\|_1 \leq \ve' \qquad \text{ for all }i=0,1,\hdots,r.$$
The new monotone path $\pi$ is constructed from $\pi'$ by insuring that either
$$0 \leq \Big<L_n(x_i)-L_n(x_{i-1}),e^k\Big> \leq  \Big<L_n(x'_i)-L_n(x'_{i-1}),e^k\Big> $$
or
$$ \Big<L_n(x'_i)-L_n(x'_{i-1}),e^k\Big>  \leq  \Big<L_n(x_i)-L_n(x_{i-1}),e^k\Big>  \leq 0$$
for $i=2,\hdots,r$ and each coordinate $k \in \{1,2,\hdots,q\}$.

\medskip
\noindent
Then
{\footnotesize
$$\left|{\sum\limits_{k=1}^q  \sum\limits_{i=1}^r \Big| \Big<L_n(x_i) - L_n(x_{i-1}), \nabla g_k^{H, \beta}(L_n( x_{i-1})) \Big> \Big| \over  \|L_n(\sigma)-L_n(\tau)\|_1 }
-{\sum\limits_{k=1}^q  \sum\limits_{i=1}^r \Big| \Big<L_n(x'_i) - L_n(x'_{i-1}), \nabla g_k^{H, \beta}(L_n( x'_{i-1})) \Big> \Big| \over  \|L_n(\sigma_\beta)-L_n(\tau)\|_1 }\right| \leq C'' r \ve' / \ve$$
}
\noi
for a fixed constant $C''>0$. Noticing that $~r\ve'/\ve \leq \delta /M~$ as $~r \leq 1/\ve$, and taking $M$ large enough, we obtain
$${\sum\limits_{k=1}^q  \sum\limits_{i=1}^r \Big| \Big<L_n(x_i) - L_n(x_{i-1}), \nabla g_k^{H, \beta}(L_n( x_{i-1})) \Big> \Big| \over  \|L_n(\sigma)-L_n(\tau)\|_1 } \leq 1-\delta/4.$$

\medskip
\noindent
Thus equation (\ref{eqn:meandist}) will imply 

{\small
\beas
\mathbb{E}_{\sigma, \tau}[d(X,Y)] & \leq & d(\sigma,\tau) \left[ 1-{1 \over n}\left(1-{\sum\limits_{k=1}^q  \sum\limits_{i=1}^r \Big| \Big<L_n(x_i) - L_n(x_{i-1}), \nabla g_k^{H, \beta}(L_n( x_{i-1})) \Big> \Big|+c \, \varepsilon^2 \over  \|L_n(\sigma)-L_n(\tau)\|_1 }\right) \right] \\
& \leq &  d(\sigma,\tau) \left[ 1 -{1 \over n}\big(1- (1-\delta/4)-\delta/20 \big) \right]\\
& = &  d(\sigma,\tau) \left[ 1 -{1 \over n}\delta/5 \right]
\eeas
}
as $~{c \, \varepsilon^2 \over  \|L_n(\sigma)-L_n(\tau)\|_1 } ~\leq ~c \, \ve ~\leq ~\delta/20~$ for 
$\ve$ small enough.\\

\bigskip
\noindent
{\bf Case II.} Suppose $L_n(\tau)=z$ and $L_n(\sigma)=w$, where $\|z-z_\beta\|_1 < \ve$ and $\|w-z_\beta\|_1< \ve'$.

\medskip
\noindent
Similarly to (\ref{eqn:grad}), equation (\ref{eqn:true}) implies for $n$ large enough,
\beas
\mathbb{E}[d(X,Y)] &  \leq & d(\sigma,\tau) \cdot \left[1-{1 \over n}\left(1-{\|g^{H,\beta}\big(L_n(\sigma)\big)-g^{H,\beta}\big(L_n(\tau)\big)\|_1 \over  \|L_n(\sigma)-L_n(\tau)\|_1} \right)  \right] + O\left( {1 \over n^2}\right)\\
&  \leq & d(\sigma,\tau) \cdot \left[1-{\gamma \over n} \right] + O\left( {1 \over n^2}\right)\\
&  \leq & d(\sigma,\tau) \cdot \left[1-{\gamma \over 2n} \right] 
\eeas
 by Condition \ref{ass:localbound} (see also discussion following Condition \ref{ass:localbound}).
\ep

We now state and prove the main theorem of the paper that yields sufficient conditions for rapid mixing of the Glauber dynamics of the class of statistical mechanical models discussed.

\begin{thm}
\label{thm:main}
Suppose $H(z)$ and $\beta>0$ are such that  Condition \ref{RS} and Condition \ref{ass:localbound} are satisfied. Then the mixing time of the Glauber dynamics satisfies
$$t_{mix} = O(n \log n)$$
\end{thm}

\begin{proof}  
Let $\ve'>0$ and $\alpha>0$ be as in Lemma \ref{lemma:contract}.
Let $(X^t, Y^t)$ be a coupling of the Glauber dynamics such that $X^0 \overset{dist}{=} P_{n, \beta}$, the stationary distribution.  
Then, for sufficiently large $n$,
\beas
\lefteqn{ \| P^t(Y^0, \cdot) - P_{n, \beta} \|_{\mbox{{\small TV}}} } \\
& \hsp \hsp \leq & P \{ X^t \neq Y^t \} \\
& \hsp \hsp = & P \{ d(X^t, Y^t ) \geq 1 \} \\
& \hsp \hsp \leq & \mathbb{E} [ d(X^t,Y^t)] \\
& \hsp \hsp = & \mathbb{E} [ {\scriptstyle \mathbb{E}[d(X^t,Y^t)~|X^{t-1},Y^{t-1}]} ]\\
& \hsp \hsp \leq & \mathbb{E} [ {\scriptstyle \mathbb{E}[d(X^t,Y^t)~|X^{t-1},Y^{t-1}]}~|~\|L_n(X^{t-1}) - z_\beta\|_1 < \varepsilon' ] \cdot P\{\|L_n(X^{t-1}) - z_\beta\|_1<\varepsilon' \} \\
& & + n P\{\|L_n(X^{t-1}) - z_\beta\|_1 \geq \varepsilon' \} .
\eeas

\noindent
By Lemma \ref{lemma:contract}, we have
\bea
\label{eqn:meandistcontract}
\lefteqn{ \mathbb{E} [ {\scriptstyle \mathbb{E}[d(X^t,Y^t)~|X^{t-1},Y^{t-1}]}~|~\|L_n(X^{t-1}) - z_\beta\|_1 < \varepsilon' ] } \nonumber \\
& \hsp \hsp \hsp & \leq e^{-\alpha /n} \mathbb{E} [ d(X^{t-1},Y^{t-1})~|~\|L_n(X^{t-1}) - z_\beta\|_1 < \varepsilon' ] 
\eea

\noi
By iterating (\ref{eqn:meandistcontract}), it follows that 
\beas
\lefteqn{ \| P^t(Y^0, \cdot) - P_{n, \beta, K} \|_{\mbox{{\small TV}}} } \\
& \hsp \hsp  \leq & e^{-\alpha /n} \mathbb{E} [ d(X^{t-1},Y^{t-1})~|~\|L_n(X^{t-1}) - z_\beta\|_1 < \varepsilon' ] \cdot P\{\|L_n(X_{t-1}) - z_\beta\|_1<\varepsilon' \} \\
& & + n P\{\|L_n(X^{t-1}) - z_\beta\|_1 \geq \varepsilon' \} \\
& \hsp \hsp  \leq & e^{-\alpha /n} \mathbb{E} [ d(X^{t-1},Y^{t-1})]+nP\{\|L_n(X^{t-1}) - z_\beta\|_1 \geq \varepsilon' \} \\
& \hsp \hsp \vdots & \vdots\\
& \hsp \hsp \leq & e^{-\alpha t/n} \mathbb{E} [ d(X^0,Y^0)] +n \sum_{s=0}^{t-1} P\{\|L_n(X^s) - z_\beta\|_1 \geq \varepsilon' \} \\
& \hsp \hsp = & e^{-\alpha t/n} \mathbb{E} [ d(X^0,Y^0)]+n t P_{n, \beta}\{\|L_n(X^0) - z_\beta\|_1 \geq \ve'\} \\
& \hsp \hsp \leq & n e^{-\alpha t/n}+ntP_{n, \beta}\{\|L_n(X^0) - z_\beta\|_1 \geq \varepsilon' \} .
\eeas

\noi
We recall the LDP limit (\ref{eqn:ldplimit}) for $\beta$ in the single phase region $B$, 
\[ P_{n, \beta}\{ L_n(X^0) \in dx \} \Longrightarrow \delta_{z_\beta} \hsp \mbox{as $n \goto \infty$}. \]
Moreover, for any $\gamma'>1$ and $n$ sufficiently large, by the LDP upper bound (\ref{eqn:ldpbound}), we have
\beas
\|P^t(Y^0, \cdot) - P_{n, \beta} \|_{\mbox{{\small TV}}} & \leq & n e^{-\alpha t/n}+n t P_{n, \beta}\{\|L_n(X^0) - z_\beta\|_1 \geq \varepsilon' \}  \\
&  <  & n e^{-\alpha t/n} +t n e^{-{n \over \gamma'}I_{\beta }(\varepsilon')} .
\eeas
For $t = {n \over \alpha} (\log n + \log (2/\ve'))$, the above right-hand side converges to $\ve'/2$ as $n \goto \infty$. 
\ep

\bigskip

\section{Aggregate Path Coupling applied to the Generalized Potts Model} \label{GCWP} 

\medskip

In this section, we illustrate the strength of our main result of section \ref{sec:main}, Theorem \ref{thm:main}, by applying it to the generalized Curie-Weiss-Potts model (GCWP), studied recently in \cite{JKRW}.  The classical Curie-Weiss-Potts (CWP) model, which is the mean-field version of the well known Potts model of statistical mechanics \cite{Wu} is a particular case of the GCWP model with $r=2$.  While the mixing times for the CWP model has been studied in \cite{CDLLPS}, these are the first results for the mixing times of the GCWP model. Moreover, the application of our methods gives a significantly shorter derivation for the region of rapid mixing than the one used for the CWP model in \cite{CDLLPS}, where the result in \cite{CDLLPS} is part of a complete analysis of the CWP model that includes cut-off.

\medskip

Let $q$ be a fixed integer and define $\Lambda = \{ e^1, e^2, \ldots, e^q \}$, where $e^k$ are the $q$ standard basis vectors of $\R^q$.  A  {\it configuration} of the model has the form $\omega = (\omega_1, \omega_2, \ldots, \omega_n) \in \Lambda^n$.  We will consider a configuration on a graph with $n$ vertices and let $X_i(\omega) = \omega_i$ be the {\it spin} at vertex $i$.  The random variables $X_i$'s for $i=1, 2, \ldots, n$ are independent and identically distributed with common distribution $\rho$.  

For the generalized Curie-Weiss-Potts model, for $r \geq 2$, the interaction representation function as in Definition \ref{defn:irf}, has the form
\[ H^r(z) = - \frac{1}{r} \sum_{j=1}^q z_j^r \]
and the generalized Curie-Weiss-Potts model is defined as the Gibbs measure
\be 
\label{eqn:GCWPgibbs}
P_{n, \beta, r} (B) = \frac{1}{Z_n(\beta)} \int_B \exp \left\{  -\beta n \, H^r\left(L_n(\omega) \right) \right\} dP_n 
\ee
where $L_n(\omega)$ is the empirical measure defined in (\ref{eqn:empmeasure}).

\medskip

In \cite{JKRW}, the authors proved that there exists a phase transition critical value $\beta_c(q, r)$ such that in the parameter regime $(q, r) \in \{2\} \times [2,4]$, the GCWP model undergoes a continuous, second-order, phase transition and for $(q,r)$ in the complementary regime, the GCWP model undergoes a discontinuous, first-order, phase transition.  This is stated in the following theorem.

\begin{thm}[Generalized Ellis-Wang Theorem]
\label{thm:EW}
Assume that $q \geq 2$ and $r \geq 2$.  Then there exists a critical temperature $\beta_c(q,r) > 0$ such that in the weak limit
\[ \lim_{n \goto \infty} P_{n, \beta, r} (L_n \in \cdot) = \left\{ \begin{array}{ll} \delta_{1/q(1, \ldots, 1)} & \mbox{if} \ \beta < \beta_c(q,r) \vspace{.1in} \\   \frac{1}{q} \sum_{i=1}^q \delta_{u(\beta, q, r) e^i + (1-u(\beta, q, r))/q(1, \ldots, 1)} & \mbox{if} \ \beta > \beta_c(q,r) \end{array} \right. \]
where $u(\beta, q, r)$ is the largest solution to the so-called mean-field equation 
\[ u = \frac{1 - \exp( \Delta(u))}{1+(q-1) \exp (\Delta (u))} \]
with $\Delta (u) :=-{\beta \over q^{r-1}}\big[(1+(q-1)u)^{r-1}-(1-u)^{r-1} \big]$.  Moreover, for $(q, r) \in \{2\} \times [2,4]$, the function $\beta \mapsto u(\beta, q, r)$ is continuous whereas, in the complementary case, the function is discontinuous at $\beta_c(q,r)$.
\end{thm}

\medskip

For the GCWP model, the function $g_{\ell}^{H, \beta}(z)$ defined in general in (\ref{eqn:gell}) has the form
\[ 
g_{k}^{H, \beta}(z) = \left[ \partial_{k} \Gamma \right] (\beta \nabla H(z)) = \left[ \partial_{k} \Gamma \right] (\beta z) = {e^{\beta z_k^{r-1}} \over e^{\beta z_1^{r-1}}+ \ldots +e^{\beta z_q^{r-1}}}.
\]
For the remainder of this section, we will replace the notation $H, \beta$ and refer to $g^{H, \beta}(z)=\big(g_1^{H, \beta}(z),\hdots,g_q^{H, \beta}(z)\big)$ as simply $g^r(z)=\big(g_1^r(z),\hdots,g_q^r(z)\big)$.  As we will prove next, the rapid mixing region for the GCWP model is defined by the following value.

\bea
\label{eqn:mixcrit} 
\beta_s(q, r) := \sup \left\{ \beta \geq 0 : g_k^r(z) < z_k  \ \mbox{for all $z \in \mathcal{P}$ such that}  \ z_k \in (1/q, 1] \right\}. 
\eea

\medskip

\begin{lemma} 
\label{lemma:CritIneq}
If $\beta_c(q,r)$ is the critical value derived in \cite{JKRW} and defined in Theorem \ref{thm:EW}, then
$$\beta_s(q,r) \leq \beta_c(q,r)$$ 
\end{lemma}
\begin{proof}
We will prove this lemma by contradiction. Suppose $\beta_c(q,r) < \beta_s(q,r)$. Then there exists $\beta$ such that 
$$\beta_c(q,r) < \beta < \beta_s(q,r).$$
Then, by Theorem \ref{thm:EW}, since $\beta_c(q,r) <\beta$, there exists $u>0$ satisfying the following inequality
\begin{equation}\label{ineq:du}
u < {1-e^{\Delta(u)} \over 1+(q-1)e^{\Delta(u)}},
\end{equation}
where $~\Delta(u):=-{\beta \over q^{r-1}}\big[(1+(q-1)u)^{r-1}-(1-u)^{r-1} \big]$. Here, the above inequality (\ref{ineq:du}) rewrites as 
\begin{equation}\label{ineq:du2}
e^{\Delta(u)}=\exp\left\{\beta\left[\left({1-u \over q} \right)^{r-1}-\left({1+(q-1)u \over q} \right)^{r-1}\right]\right\} ~<~ {1-u \over (q-1)u+1}.
\end{equation}
Next, we substitute $\lambda=(1-u){q-1 \over q}$ into the above inequality (\ref{ineq:du2}), obtaining
\begin{equation}\label{ineq:du3}
\exp\left\{\beta\left[\left({\lambda \over q-1} \right)^{r-1}-\left(1-\lambda \right)^{r-1}\right]\right\} ~<~ {\lambda \over (1-\lambda)(q-1)}.
\end{equation}
Now, consider 
$$z=\left(1-\lambda, {\lambda \over q-1},\hdots, {\lambda \over q-1} \right).$$
Observe that $~z_1=1-\lambda=1-(1-u){q-1 \over q}={1+u(q-1) \over q}>{1 \over q}~$ as $u>0$.
Here, the inequality (\ref{ineq:du3}) can be consequently rewritten in terms of the above selected $z$ as follows
$$z_1=1-\lambda ~<~{e^{\beta(1-\lambda)^{r-1}} \over e^{\beta(1-\lambda)^{r-1}}+(q-1)e^{\beta\big({\lambda \over q-1}\big)^{r-1}}}=g_1^r(z),$$
thus contradicting $\beta < \beta_s(q,r)$. Hence $~\beta_s(q,r) \leq \beta_c(q,r)$.
\end{proof}

Combining Theorem \ref{thm:EW} and Lemma \ref{lemma:CritIneq} yields that for parameter values $(q, r)$ in the continuous, second-order phase transition region $\beta_s(q,r) = \beta_c(q,r)$, whereas in the discontinuous, first-order, phase transition region, $\beta_s(q,r)$ is strictly less than $\beta_c(q,r)$. This relationship between the equilibrium transition critical value and the mixing time transition critical value was also proved for the mean-field Blume-Capel model discussed in \cite{KOT}.  This appears to be a general distinguishing feature between models that exhibit the two distinct types of phase transition.  We now prove rapid mixing for the generalized Curie-Weiss-Potts model for $\beta < \beta_s(q,r)$ using the aggregate path coupling method derived in section \ref{sec:main}.

\medskip
We state the lemmas that we prove below, and the main result for the Glauber dynamics of the generalized Curie-Weiss-Potts model, a Corollary  to Theorem \ref{thm:main}. 

\begin{lemma}\label{cwp1}
Condition \ref{uniform} and Condition \ref{RS} are satisfied for all $\beta < \beta_s(q,r)$.
\end{lemma}

\begin{lemma}\label{cwp2}
Condition \ref{ass:localbound} is satisfied for all $\beta < \beta_s(q,r)$.
\end{lemma}

\begin{cor}
\label{cor:CWPrapid} 
If $\beta < \beta_s(q,r)$, then
\[ t_{\mbox{{\em \mix}}} = O(n \log n). \]
\end{cor}
\bp
Condition \ref{RS} and Condition \ref{ass:localbound} required for Theorem \ref{thm:main} are satisfied by Lemma \ref{cwp1} and Lemma \ref{cwp2}.
\ep

\medskip
\noindent
\bp[Proof of Lemma \ref{cwp2}]
Denote $~z'=(z'_1,\hdots,z'_q)=z-z_\beta$. Then by Taylor's Theorem, we have
\bea \label{zero} 
\limsup_{z \rightarrow z_\beta} {\|g^r(z)-g^r(z_\beta)\|_1 \over \|z-z_\beta\|_1} & = & \limsup_{z \rightarrow z_\beta} {\sum\limits_{k=1}^q \left|{e^{\beta z_k^{r-1}} \over \sum\limits_{j=1}^q e^{\beta z_j^{r-1}}} -{1 \over q} \right| \over \sum\limits_{k=1}^q \left|z_k -{1 \over q} \right|} \nonumber\\
& = & \lim_{z' \rightarrow 0} \frac{ \sum\limits_{k=1}^q \left|{\beta (r-1) \left({1 \over q} \right)^{r-2} z'_k +O\big((z'_1)^2+\hdots+(z'_q)^2\big) \over q+O\big((z'_1)^2+\hdots+(z'_q)^2\big)  }  \right| }{ \sum\limits_{k=1}^q \left|z'_k  \right| } \nonumber\\
&= & \frac{\beta (r-1)}{q^{r-1}}.
\eea
Recall that  $\beta_s(q,r)\leq \beta_c(q,r)$ was shown in  Lemma \ref{lemma:CritIneq}, and $\beta_c(q,r) \leq {q^{r-1} \over r-1}$  was shown in the proof of Lemma 5.4 of \cite{JKRW}. 
Therefore, $~\beta < {q^{r-1} \over r-1}$ and the last expression above is less than $1$, and we conclude that 
\[ \limsup_{z \rightarrow z_\beta} {\|g^r(z)-g^r(z_\beta)\|_1 \over \|z-z_\beta\|_1} < 1. \qedhere \]
\ep

\medskip
\noindent
{\it Proof of Lemma \ref{cwp1}.} First, we prove that the family of straight lines connecting to the equilibrium point $z_\beta=\left(1/q,\hdots,1/q\right)$ is a neo-geodesic family as it was defined following Condition \ref{uniform}.  Specifically, for any $z = (z_1, z_2, \ldots, z_q) \in \mathcal{P}$ define the line path $\rho$ connecting $z$ to $z_\beta$ by 
\be
\label{eqn:linepath2} 
z(t) = {1 \over q} (1-t) + z \, t, \hsp 0 \leq t \leq 1 
\ee

\noindent
Then, along this straight-line path $\rho$, the aggregate $g$-variation has the form
\[ D_\rho^g (z,z_\beta) := \sum\limits_{k=1}^q\int\limits_\rho \Big| \Big<\nabla g_k^r (y), dy \Big> \Big| =  \sum\limits_{k=1}^q \int_0^1 \left| \frac{d}{dt} [g_k^r (z(t))] \right| \, dt \]

\medskip
\noindent
Next, for all $k = 1, 2, \ldots, q$ and $t \in [0,1]$, denote 
\[ z(t)_k = {1 \over q} (1-t) + z_k t \]
Then
\be
\label{eqn:coordg2} 
g_k^r(z(t)) = \frac{e^{\beta \big((1/q)(1-t) + z_k t \big)^{r-1}}}{\sum_{j=1}^q e^{\beta \big((1/q) (1-t) + z_j t \big)^{r-1}}}  
\ee
and 
\be
\label{eqn:dg}
\frac{d}{dt} \big[g_k^r(z(t)) \big] = \beta (r-1) g_k^r (z(t)) \Big[ \left(\frac{1}{q} (1-t) + z_k t \right)^{r-2} \left(z_k- \frac{1}{q} \right) - \langle z-z_\beta, g^r(z(t)) \rangle_\rho \Big] 
\ee
where $\langle z-z_\beta, g^r(z(t)) \rangle_\rho$ is the weighted inner product
\[ \langle z-z_\beta, g^r(z(t)) \rangle_\rho := \sum_{j=1}^q g_j^r(z(t)) \left(z_k- \frac{1}{q} \right)  \left(\frac{1}{q} (1-t) + z_k t \right)^{r-2} \]

\medskip
\noindent
Now, observe that for $z(t)$ as in (\ref{eqn:linepath2}) with $z \not= z_\beta$, the inner product $\langle (z-z_\beta), g^r(z(t)) \rangle_\rho$ is monotonically increasing in $t$ since
\[  \frac{d}{dt} \langle z-z_\beta, g^r(z(t)) \rangle_\rho \geq \beta (r-1) \, \mbox{Var}_{g^r}\left( \left(z_k-\frac{1}{q} \right) \left(\frac{1}{q} (1-t) + z_j t  \right)^{r-1} \right) > 0 \]
where $\mbox{Var}_{g^r}(\cdot)$ is the variance with respect to $g^r$.

So $\langle z-z_\beta, g^r(z(t)) \rangle_\rho$ begins at $\langle z-z_\beta, g^r(z(0)) \rangle_\rho =\langle z-z_\beta, z_\beta \rangle =0$ and increases for all $t \in (0,1)$.

\bigskip
\noindent
The above monotonicity yields the following claim about the behavior of $g_k^r(z(t))$ along the straight-line path $\rho$.

\begin{itemize}
\item[(a)] If $z_k \leq 1/q$, then $g_k^r(z(t))$ is monotonically decreasing in $t$.

\item[(b)] If $z_k > 1/q$, then $g_k^r(z(t))$ has at most one critical point $t_k^\ast$ on $(0,1)$.
\end{itemize}

\noindent
The above claim (a) follows immediately from (\ref{eqn:dg}) as $\langle z-z_\beta, g^r(z(t)) \rangle_\rho >0$ for $t>0$. Claim (b) also follows from (\ref{eqn:dg}) as its right-hand side, $z_k-1/q>0$ and $\langle z-z_\beta, g^r(z(t)) \rangle_\rho$ is increasing. Thus there is at most one point $t_k^\ast$ on $(0,1)$ such that $~\frac{d}{dt} \big[g_k^r(z(t)) \big] =0$.

\bigskip
\noi
Next, define 
\[ A_z = \{ k : z_k > 1/q \} \]
Then the aggregate $g$-variation can be split into 
\[ D_\rho^g (z,z_\beta) = \sum_{k \in A_z} \int_0^1 \left| \frac{d}{dt} [g_k^r(z(t))] \right| \, dt + \sum_{k \notin A_z} \int_0^1 \left| \frac{d}{dt} [g_k^r(z(t))] \right| \, dt \]
For $k \notin A_z$, claims (a) and (b) imply
\[ \int_0^1 \left| \frac{d}{dt} [g_k^r(z(t))] \right| \, dt = - \int_0^1 \frac{d}{dt} [g_k^r(z(t))] \, dt = g_k^r(z(0)) - g_k^r(z(1)) = \frac{1}{q} - g_k^r(z) \]
For $k \in A_z$, let $t_k = \max\{ t_k^\ast, 1\}$ ,where $t_k^\ast$ is defined in (b).  Then, we have 
\[ \int_0^1 \left| \frac{d}{dt} [g_k^r(z(t))] \right| \, dt = \int_0^{t_k^\ast} \frac{d}{dt} [g_k^r(z(t))] \, dt - \int_{t_k^\ast}^1 \frac{d}{dt} [g_k^r(z(t))] \, dt = 2 g_k^r(z(t_k^\ast)) - g_k^r(z) - \frac{1}{q} \]
Combining the previous two displays, we get
\beas
D_\rho^g (z,z_\beta) & = & \sum_{k \in A} \left( 2 g_k^r(z(t_k^\ast)) - g_k^r(z) - \frac{1}{q} \right) + \sum_{k \notin A} \left(\frac{1}{q} - g_k^r(z)  \right) \\
& = & 2 \sum_{k \in A} \left( g_k^r(z(t_k^\ast)) - \frac{1}{q} \right)
\eeas
Since $\beta < \beta_s$ and $k \in A_z$, we have 
\[ g_k^r(z(t_k^\ast)) < z(t_k^\ast)_k \leq z(1)_k = z_k \]
and we conclude that 
\[ D_\rho^g (z,z_\beta) < 2 \sum_{k \in A} \left( z_k - \frac{1}{q} \right) = \| z - z_\beta \|_1 \]
Thus
$${d_g(z,z_\beta) \over \| z - z_\beta \|_1} \leq {D_\rho^g (z,z_\beta) \over \| z - z_\beta \|_1} <1 \quad \text { for all } z \not=z_\beta \text{ in } \mathcal{P}.$$

\bigskip
\noi
Next, since we are dealing with the straight line segments $\rho$,
$$\limsup_{z \rightarrow z_\beta}{D_\rho^g (z,z_\beta) \over \| z - z_\beta \|_1}=\limsup_{z \rightarrow z_\beta} {\|g(z)-g(z_\beta)\|_1 \over \|z-z_\beta\|_1} <1$$
by (\ref{zero}), the Mean Value Theorem, and $H(z) \in \mathcal{C}^3$. This, in turn, guarantees the continuity required for Condition \ref{uniform}:
$$\limsup_{z \rightarrow z_\beta}{d_g(z,z_\beta) \over \| z - z_\beta \|_1} \leq \limsup_{z \rightarrow z_\beta}{D_\rho^g (z,z_\beta) \over \| z - z_\beta \|_1}<1$$
Thus Condition \ref{uniform} is proved for the CWP model. Moreover this proves that the family of straight line segments $\rho$ is a neo-geodesic family (see definition following Condition \ref{uniform}). Indeed, there is $\delta \in (0,1)$ such that 
$$\left\{\rho:~z(t)={1 \over q}(1-t)+zt, ~~z \in \mathcal{P} \right\} \quad \text{ is a } {\bf NG}_\delta \text{ family of smooth curves,}$$
 i.e. $\forall z\not= z_\beta$ in $\mathcal{P}$, and corresponding $\rho:~z(t)={1 \over q}(1-t)+zt$,
$$ {D_\rho^g (z,z_\beta) \over \| z - z_\beta \|_1} \leq 1-\delta/2 \qquad$$

\bigskip
\noi
Since the family of straight line segments $\rho$ is a neo-geodesic family ${\bf NG}_\delta$, the integrals
$$ D_\rho^g (x, z) := \sum\limits_{k=1}^q\int\limits_{\rho} \Big| \Big<\nabla g_k^r(y), dy \Big> \Big| $$
can be uniformly approximated by the corresponding  Riemann sums of small enough step size by the Mean Value Theorem as $H(z) \in \mathcal{C}^3$ and therefore each $g_k^r(z) \in \mathcal{C}^2$.
That is, there exists a constant $C>0$ that depends on the second partial derivatives of $g^r(z)=\big(g_1^r(z),\hdots,g_q^r(z)\big)$, such that for $\varepsilon>0$ small enough, the curve  $\rho={1 \over q}(1-t)+zt$ in the family  ${\bf NG}_\delta$ that connects $z_\beta$ to $z$ satisfies
$$\left| \sum_{k=1}^q  \sum_{i=1}^r \Big| \Big<z_i - z_{i-1}, \nabla g_k^r(z_{i-1}) \Big> \Big|-D_\rho^g (z,z_\beta) \right| < C r \ve^2
\qquad \forall z \in \mathcal{P} \text{ s.t. } \|z-z_\beta\|_1 \geq \varepsilon $$
for a sequence of points $z_0=z_\beta,z_1,\hdots,z_r=z \in \mathcal{P}$ interpolating $\rho$ such that
$$\varepsilon \leq \|z_i - z_{i-1}\|_1  < 2\varepsilon \quad \text{ for } i=1,2,\hdots, r.$$
Hence
$${\sum\limits_{k=1}^q  \sum\limits_{i=1}^r \Big| \Big<z_i - z_{i-1}, \nabla g_k^r(z_{i-1}) \Big> \Big| \over \|z-z_\beta\|_1} \leq 1-\delta/2+C\ve \leq 1-\delta/3$$
for $\ve \leq \delta/(6C)$. This concludes the proof of Condition \ref{RS}. \hfill $\square$

\bigskip
Finally, the region of exponentially slow mixing $\beta> \beta_s(q,r)$ can be shown using the standard approach of bottleneck ratio similar to section 7 in \cite{KOT}.

\section*{Acknowledgments}
This work was partially supported by a grant from the Simons Foundation ($\#$284262 to Yevgeniy Kovchegov).


\bibliographystyle{amsplain}

\end{document}